\documentclass[10pt]{amsart}
\usepackage{amssymb, amsmath}
\usepackage{graphics}
\usepackage{latexsym}
\usepackage{amsmath}
\usepackage{amssymb,amsthm,amsfonts}
\usepackage{amscd}
\usepackage[arrow, matrix, curve]{xy}
\usepackage{syntonly}
\usepackage{yfonts}
\usepackage{tikz-cd}
\usepackage{filecontents}
\usepackage{mathtools}
\usepackage{stmaryrd}
\usepackage{MnSymbol}
\DeclareMathOperator{\aut}{Aut}

\DeclareMathOperator{\jac}{Jac} 

\DeclareMathOperator{\nm}{Nm}

\DeclareMathOperator{\im}{Im}

\DeclareMathOperator{\en}{End}
\DeclareMathOperator{\alb}{Alb}
\theoremstyle{plain}
\newtheorem{thm}{Theorem}[section]
\newtheorem{theorem}[thm]{Theorem}

\newtheorem{proposition}[thm]{Proposition}

\theoremstyle{definition}

\newtheorem{remark}[thm]{Remark}

\newtheorem{definition}[thm]{Definition}

\numberwithin{equation}{thm}

\newcommand{\sA}{{\mathcal A}}
\newcommand{\sB}{{\mathcal B}}
\newcommand{\sC}{{\mathcal C}}

\newcommand{\sK}{{\mathcal K}}
\newcommand{\sL}{{\mathcal L}}

\newcommand{\sU}{{\mathcal U}}
\newcommand{\sV}{{\mathcal V}}

\newcommand{\sX}{{\mathcal X}}

\newcommand{\sZ}{{\mathcal Z}}

\renewcommand{\P}{{\mathbb P}}

\newcommand{\Z}{{\mathbb Z}}

\newcommand{\rk}{{\rm rank}}

\newcommand{\Hom}{{\rm Hom}}

\begin{document}
	\title{Twists of Albanese varieties over function fields with large ranks}
	\author{Abolfazl Mohajer, Sajad Salami}
	\address{Johannes Gutenberg Universit\"at Mainz, Institut f\"ur Mathematik, Staudingerweg 9, 55099 Mainz, Germany}
	\address{Instítuto de Matemática e EstatísticaUniversidade Estadual do Rio do Janeiro, Brazil}
	\email{mohajer@uni-mainz.de}
	\email{sajad.salami@ime.uerj.br}
	\maketitle
	
	\begin{abstract}
		In this paper we construct abelian varieties of large Mordell-Weil rank over function fields.  We achieve this by using a generalization of the notion of Prym variety to higher dimensions and a structure theorem for the Mordell-Weil group of abelian varieties over function fields proven in our previous works. We consider abelian and dihedral covers of the projective space and apply the above results to the twists of their Albanese varieties.  
		
	\end{abstract}
\section{Introduction}
Let $A$ be an abelian variety over a field $k$ of characteristic $\geq 0$. Let $A(k)$ be the set of $k$-rational points of $A$. If $k$ is a number field or a  function field satisfying some mild conditions, then it is well-known that $A(k)$ is a finitely generated abelian group. The rank of this group is called the {\it Mordell-Weil rank} of $A$. Finding abelian varieties with large rank is an interesting problem in number theory. In this paper we investigate this problem by considering abelian and dihedral covers of algebraic varieties. Let $n\in \mathbb{N}$ be coprime to the characteristic of the field $k$  and such that $k$ contains a primitive $n$-th root of unity $\xi$. Suppose there is an embedding $G\hookrightarrow \aut(A)$ with $G$ a finite group of order $|G|=n$. Further, let $X,Y$ be smooth projective varieties over $k$ with function fields $L,K$ respectively. Let $f:X\to Y$ be a $G$-Galois covering. To such a covering we have associated in \cite{M2} an abelian varietiy $P(X/Y)$, the so-called \emph{Prym variety}, that generalizes the notion of Prym variety in the case of coverings of curves. The group $G$ is also the Galois group of the Galois field extension $L|K$. By the results of \cite{BS, Haz1}, the twist of $A$ by the extension $L|K$ is equivalent to the twist by the 1-cocyle $a=(a_g)\in Z^1(G,\aut(A))$ given by $a_g=g$, where in the notation $a_g$, $g$ is viewed as a group element and on the right side as an automorphism of $A$ corresponding to $g\in G$ (i.e., we identify $g$ with its image $g\in G\hookrightarrow \aut(A)$). By the results of \cite{M2} there is an isomorphism for $A_a(K)$ in terms of the Prym variety. Furthermore, \cite{M2} gives a description of the Prym variety of the $m$-times self product $\prod_i f$ of the $G$-cover $f:X\to Y$ with itself, which generalizes the results of \cite{Sal1} and \cite{Haz1} for cyclic and double covers. Such a product yields a $G$-Galois field extension (by considering the function fields of $\prod_i X$ and $\prod_i Y$) and hence a twist of the abelian varieties defined over the function field of these varieties.
In this paper, on the one hand we consider abelian covers $X\to\P^{\ell}$ of the $\ell$-dimensional projective space and using the above results show that the rank of the twisted Albanese variety $\widetilde{\alb(X)}$ becomes arbitrarily large if one takes $m$ to be large enough. Furthermore, we consider non-abelian dihedral covers $C\to\P^1$ of the projective line and show that the rank of the twist of the Jacobian $\jac_a(C)$ can be made arbitrarily large. In fact one can generalize this result also to general dihedral covers of higher-dimensional projective spaces $\P^{\ell}$. Our results also more generally apply to metacyclic covers of the projective spaces.

\section{The Prym variety of Galois coverings and twists}
Let $V$ be a smooth projective variety over a field $k$. The \emph{Albanese variety} of $V$ is  the initial object for the morphisms from $V$ to abelian varieties. If $V$ is defined over a field of characterictic zero, then it is the following abelian variety,
\begin{equation}\label{Def albanese}
	\alb(V)=H^0(V,\Omega^1_V)^*/H_1(V,\mathbb{Z})
\end{equation}
Let $X,Y$ be smooth projective varieties with $f:X\to Y$ a Galois covering with Galois group $G$. In other words there is a finite group $G$ with $|G|=n$ that has a $G$-linearized action on $X$ such that $Y\coloneqq X/G$ and $f:X\to Y$ is the  quotient map. By the universal property of the albanese variety there is an induced action of $G$ on $\alb(X)$. Let ${(\alb(X)^G)}^0$ be the largest abelian subvariety of $\alb(X)$ fixed (pointwise) under this action. Equivalently, ${(\alb(X)^G)}^0$ is the connected component of the identity of the subvariety $\alb(X)^G$ of fixed points of $\alb(X)$ under the action of $G$. \par Let $g\in G$ be an element of the group $G$.  The automorphism of $\alb(X)$ induced by $g$ is also denoted also by $g$ throughout the manuscript. Using this notation let $\nm G:\alb(X)\to\alb(X)$ be the \emph{norm endomorphism} of $\alb(X)$ given by $\nm G:=\sum_{g\in G}g$. This definition is motivated by \cite{RR}, Prop 3.1 in the case of curves and extends it to higher dimensions. \par The following definition of Prym variety for the covering $f$ is given in \cite{M2}, Definition 2.1.  
\begin{definition}\label{prymdef}
	Let $f:X\to Y$ be $G$-Galois covering of smooth projective varieties over a field $k$. The Prym variety $P(X/Y)$ of the covering $f$ is defined to be
	\begin{equation}
		P(X/Y):=\frac{\alb(X)}{{(\alb(X)^G)}^0}
	\end{equation}
	In particular $P(X/Y)$ is (isogeneous to) the complementary abelian subvariety in $\alb(X)$ of (the abelian subvariety) ${(\alb(X)^G)}^0$ .
\end{definition}
For the classical case of Prym varieties of double coverings of curves we refer to \cite{B}. The following result has been proven in \cite{M2}, Prop 2.2 and gives an equivalent description of the Prym variety and shows furthermore that it coincides with the Prym variety of covers of curves (see \cite{B}, \cite{LO} and \cite{RR}) up to isogeny.
\begin{proposition}\label{prym properties}
	With the above notation,
	$$P(X/Y)=\frac{\alb(X)}{\im\nm G} \sim {(\ker\nm G)}^0.$$
\end{proposition}
\begin{remark}\label{prym for curves}
	The albanese variety of curves coincides with the Jacobian variety. Hence if $X$ and $Y$ are curves, the Prym variety $P(X/Y)$ coincides with the Prym variety for covers of curves, see \cite{LO} and \cite{M1}. In this case there are two homomorphisms: the norm homomorphism $\nm_f:\jac(X)\to\jac(Y)$ and the pull-back homomorphism $f^*:\jac(Y)\to\jac(X)$ and it holds that: $\nm G=f^*\circ\nm_f$, so that $P(X/Y)={(\ker\nm_f)}^0={(\ker\nm G)}^0$, see \cite{RR}, Prop 3.1.
\end{remark}

Given a $G$-Galois covering $f:X\to Y$ defined over $k$, one can form the $m$-times self product $ f_m: \prod_{i=1}^m X \to \prod_{i=1}^m Y$. Then the diagonal embedding $G\hookrightarrow \prod_i G:=G\times\cdots\times G$ gives a subgroup of $\prod_i G$ isomorphic to $G$. We denote this subgroup by $\tilde{G}$. This gives rise to an intermediate Galois covering $f'_m:\prod_{i=1}^m X\to(\prod_{i=1}^m X)/\tilde{G}$.  
Let us write $\mathcal{X}=\prod_{i=1}^m X$ and $\mathcal{Y}=(\prod_i X)/\tilde{G}$.
We are interested in the Prym variety $P(\mathcal{X}/\mathcal{Y})$. We have
\begin{proposition}[\cite{M2}, Prop 2.5]\label{prym product}
	With the above notation, there is an isogeny
	\begin{equation}
		P(\mathcal{X}/\mathcal{Y})\sim_k \prod_{i=1}^m P(X_i/ Y_i)
	\end{equation}
\end{proposition}

We recall from \cite{BS} the notion of twist of a smooth projective variety. Suppose $k^{\prime}/k$ is a $G$-Galois extension of fields. Let $\sA$ be a $G$-set and suppose $E$ is another $G$-set which is also a left $\sA$-set. Let $a=(a_g)\in Z^1(G,\aut(\sA))$ be a 1-cocycle of $\sA$. For any $g\in G$ and $x\in E$, we denote by $^gx$ the left action of $g$ on $x$. The $G$-set $E$ with this action of $G$ is denoted by $E_a$ and is called the twist of $E$ by $a$. If $X$ is a projective smooth variety over $k$, we denote by $\aut(X)$ the automorphism scheme of $X$ and let $a=(a_g)\in Z^1(G,\aut(X))$. There exists a projective variety $\widetilde{X}$ over $k$ and $k^{\prime}$-isomorphism $f^{\prime}:X\otimes_k k^{\prime}\to \widetilde{X}\otimes_k k^{\prime}$ such that $^g(f^{\prime})=f^{\prime}\circ a_g$ for every $g\in G$. We denote the variety $\widetilde{X}$ by $X_a$ and call it \emph{the twist} of $X$ by $a$. Of course the two notion of twist introduced above are compatible in the sense that the map $f^{\prime}:X(k^{\prime})\to \widetilde{X}(k^{\prime})$ induces an isomorphism between $X(k^{\prime})_a\cong X_a(k^{\prime})(=\widetilde{X}(k^{\prime}))$.\par  In this paper, we are mainly interested in the case where we have an abelian variety $A$ and its twist $A_a$ by a field extension $L|K$ or equivalently a twist by the 1-cocycle $a=(a_g)\in Z^1(G,\aut(A))$ given by $a_g=g$. For such a $1$-cocycle, let $\alb(a):=\alb(a_g)$ for $g\in G$ which satisfies also $1$-cocycle conditions, so that $\alb(a)\in Z^1(G,\aut(\alb(X)))$.

The relation between Albanese variety and its twists is given by the following proposition which is proven in \cite{Sal2}, Proposition 3.2.

\begin{proposition} \label{twist alb}
	The twist $\alb(X)_{\alb(a)}$ of $\alb(X)$ by the $1$-cocycle $\alb(a)$ is $K$-isomorphic to $\alb(X_a)$. Equivalently $\widetilde{\alb(X)}$, the twist of $\alb(\sX)$, is $K$-isomorphic to $\alb(\tilde{X})$.    
\end{proposition}

\section{The Main results}

Given an integer $n \geq 2$, let  $\pi : \sX' \rightarrow \sX$  be a $G$-Galois cover of smooth projective varieties of degree $n$, both as well as $\pi$ defined over a field $k$. Denote by  $\sK$ and   $\sL$  the function fields of $\sX$ and $\sX'$ respectively.
Assume that $\sA$ is an abelian variety such that there is an embedding $G\hookrightarrow\aut(\sA)$, and
let $\sA[n](k)$ be the  group of $k$-rational $n$-division points on $\sA$.
Define  $\tilde{\sA}$ to be the twist of $\sA$ by the $G$-extension $\sL|\sK$,
or equivalently, by  the  $1$-cocycle $a=(a_u) \in \sZ^1(G, \aut(\sA))$, where $a_{u}=u$ for  $u\in G$ (we denote the image of $u$ under $G\hookrightarrow\aut(\sA)$ again by $u$) and $G$ is the Galois group of the extension $\sL|\sK$.
The  following theorem  describes the structure of the Mordell-Weil  group of $\sK$-rational points on the twist $\tilde{\sA}$. This is Proven in \cite{M2}, Theorem 2.4 and generalizes the main results of \cite{Haz, Sal1, Wang}.
\begin{theorem}\label{prym-abelian}
	Notation being as above, we assume that there exists a  $k$-rational  
	point  on $ \sX'.$  Then, as an isomorphism of  abelian groups,  we have:
	\begin{equation}\label{rational isom}
		\tilde{\sA}(\sK) \cong \Hom_{k} (P({\sX'/\sX}),\sA) \oplus \sA [n](k).
	\end{equation}
	Moreover, if  $P({\sX'/\sX})$ is $k$-isogenous to $ \sA^m \times \sB$ for some integer $m>0$ and $\sB$ is an  abelian variety defined over $k$ so that none of its  irreducible components is $k$-isogenous to $\sA$,  then $$\rk(\tilde{\sA}(\sK))= m\cdot \rk (\en_{k}(\sA)).$$
\end{theorem}

We apply Theorem \ref{prym-abelian} in the following situations. First, we consider abelian Galois covers $X\to\mathbb{P}^{\ell}$ with Galois group $G$, where for an integer $\ell \geq 1$, we denote by  ${\mathbb P}^\ell$ the $\ell$-dimensional projective space over $\bar k$ an algebraically closed field containing the number field $k$. For an integer $m\geq 1$, we define $U_m$ to be the $m$-times self-fiber product of $X$ over $k$ and we let $V_m$ be the quotient of $U_m$ by $G$. Considering a non-singular model for $X$, say $\sX$, we obtain  the non-singular models $\sU_m$ and $\sV_m$ for $U_m$ and $V_m$, respectively.
Denoting by $\sK$ and $\sL$, the function field of $\sU_m$ and $\sV_m$, respectively,  we obtain the following theorem which generalizes
Theorem 1.2 of \cite{Sal1} and Theorem 1.1 of \cite{Sal2}.

\begin{theorem} \label{main1}
	Notation being as above,  we  assume that there exists a $k$-rational point on $X$ and hence on $\sX$. Then, as an isomorphism of  abelian groups, we have:
	$$\widetilde{\alb(\sX)}(\sK) \cong \alb(\tilde{\sX})(\sK) \cong  \big (\en_k( \alb(\sX))\big)^m \oplus \alb(\sX)[n](k))$$
	and hence, $$\rk (\widetilde{\alb(\sX)}(\sK))= m\cdot \rk (\en_{k}(\alb(\sX))),$$
	where $\en_{k}(*)$ is  the ring of endomorphisms over $k$ of its origin $(*)$ and $\rk(\en_{k}(*))$ denotes its rank as a $\Z$-module.
\end{theorem}
\begin{proof}
	Without loss of generality, we may identify   $\bar k$  with $\mathbb C$ the field of complex numbers.
	By the main result of G. Yun in \cite{Yun}, any finite abelian Galois cover $X\to\mathbb{P}^{\ell}$ with $\ell \geq 1$ is determined by a collection of equations of the  form $w^{n_j}_j=f_j(x_1,\cdots, x_\ell)$ for $j=1,2,\dots, r$ for some integer $r\geq 1$ where
	$n_1, \cdots, n_r$ are integers such that $n_j$ divides $n_{j+1}$ for $j=1,\cdots, r-1$.
	Hence, the fibered product $U_m=X^{(1)} \times_k \cdots \times_k X^{(m)}$  for any integer $m\geq 1$, where $X^{(i)}$ is a copy of $X$, can be defined by  the affine equations $w_{i,j}^{n_j}=f_j (x_{i ,1 },\cdots, x_{i, \ell })$ for   $1 \leq i \leq m$ and $1 \leq j \leq r$.
	Now, consider the fibered product
	$\sU_m=\sX_n^{(1)} \times_k \cdots \times_k \sX_n^{(m)}$, where $\sX_n^{(i)}$ is a non-singular
	projective model of $X^{(i)}$ for
	each $1\leq i \leq m$. Note that $\sU_m$  can be viewed as a non-singular model of $U_m$.
	Denote by $L$ and $\sL$ the function fields of $U_m$ and $\sU_m$, respectively.
	Let $V_m$ and $\sV_m$ be the quotient of $U_m$ and $\sU_m$ by $G$ and $\bar{G}$, the image of $G$ in $\prod_{i=1}^m G$,
	and denote by $K$ and $\sK$ the function fields of $V_m$ and $\sV_m$, respectively.
	Then, both of the extensions $L|K$ and $\sL| \sK$ are finite abelian extension of order $n$.
	Indeed, we have $L\subset k(x_{i, e}, w_{i, j}\mid 1 \leq e \leq \ell,\  1 \leq i \leq m, \ 1 \leq j \leq r)$
	where  $x_{i, e}$'s are independent transcendental  variables and each $w_{i,j}$
	satisfy the above equations.
	Then, by definition, $K$ is the subfield of $G$-invariant elements of $L$, i.e.,
	$$K=L^G\subseteq  k ( x_{i, e},  w_{1, j}^{n_1-1}\cdot w_{2,j} , \ldots, w_{1,j}^{n_1-1}\cdot w_{m, j}).$$
	Since $(w_{1,j}^{n_1-1}\cdot w_{i+1, j})^{n_j}=f_j^{n_j-d_j}(x_{1, 1},\cdots, x_{1, \ell})^{n-1}f_j(x_{i+1, 1}, \cdots, x_{i+1, \ell})$  for $1 \leq i \leq  m-1$,  where $d_j= n_j/n_1$ for $1\leq j \leq r$, so by defining  $z_{i,j}:=w_{1,j}^{n_1-1}\cdot w_{i+1, j}$   the variety  $V_m$ can be expressed by the equations
	\begin{equation}
		\label{ho1}
		z_{i,j}^{n_j}=f_j(x_{1, 1},\cdots, x_{1, \ell})^{n_j- d_j}f_j(x_{i+1, 1}, \cdots, x_{i+1, \ell}) \ (i=1,\ldots, m-1, \ j=1,\cdots, r).
	\end{equation}
	Thus, the extension $L|K$ is  an abelian $G$-extension determined by
	$w_{1,j}^{n_1}=f_j(x_{1, 1},\cdots, x_{1, \ell})$, i.e.,
	$$L=K(w_{1,j}\mid j=1,\dots, r)\subseteq k(x_{i,e},    z_{i, j})(w_{1,j}\mid i=1,\ldots, m-1, \ j=1,\dots, r).$$
	We remark that the abelian field extension $\sL|\sK$ can be determined by considering the homogenization of the equations $w_{1,j}^{n_1}=f_j(x_{1, 1},\cdots, x_{1, \ell})$. Furthermore, the variety $\sV_m$ can be expressed by the homogenization of equations (\ref{ho1}).
	Let  $\tilde{X}$   and $\tilde{\sX}$   be the twists  of $X$  and  $\sX$   by the extensions $L|K$ and $\sL|\sK$, respectively.
	Then, as in Example 2.6 of \cite{M2}, one can check that the twist $\tilde{X}$ is given by the following set of affine equations
	\begin{equation}
		\label{ho2}
		f_j(x_{1, 1}, \cdots, x_{1,\ell})^{d_j}  z_{1,j}^{n_j}=f_j(x_{i+1, 1}, \cdots, x_{i+1,\ell}),  (i=1,\ldots, m-1, \ j=1,\dots, r)
	\end{equation}
	and  $\tilde{\sX}$ can be determined by its homogenization.
	Moreover, it is  easy to  check that $\tilde{X}$ contains the $m$   $K$-rational points:
	\begin{equation}
		\label{Kpoint}  
		P_1:=\left( x_{1, 1},\dots, x_{1,\ell},  1\right),
		\text{and} \ P_{i+1}:=(x_{i+1, 1}, \dots,  x_{i+1, \ell},  w_{i+1,j}/w_{1, j}),\ (1\leq i\leq  m-1).
	\end{equation}
	Let us  denote by $\tilde{P}_i$ the point corresponding to $P_i$'s on $\tilde{\sX}$.
	Now, by  applying Proposition \ref{prym properties} to  the  $G$-cover $\pi: \sX\rightarrow \P^\ell$, we obtain
	$$P({\sX^{(i)}/\P^\ell}) =\frac{\alb(\sX^{(i)})}{\im\nm G}.$$
	Since $\alb(\P^\ell)=0$, we have  $P({\sX^{(i)}/\P^\ell})=\alb(\sX^{(i)})=\alb(\sX)$ for  $i=1, \ldots, m$ , and hence
	using  Proposition \ref{prym product}, we get the following  $k$-isogeny of abelian varieties
	\begin{equation}
		\label{prym}
		P({\sU_m/\sV_m}) \sim_k \prod_{i=1}^m P({\sX^{(i)}/\P^\ell}) = \alb(\sX)^m.
	\end{equation}
	Now, let us consider the $1$-cocycle  $a=(a_g) \in \sZ^1(\bar{G}, \aut(\alb(\sX)))$ defined  by $a_{g}=g$. Using Proposition \ref{twist alb}, we conclude that
	$\widetilde{\alb(\sX)}\sim_\sK \alb(\tilde{\sX})$.
	Thus, by  Theorem \ref{prym-abelian}  for $\sX'=\sU_m$,  $\sX=\sV_m$, and $\sA =\alb(\sX)$, we get:
	\begin{align*}
		\widetilde{\alb(\sX)} (\sK)  & \cong  \Hom_{k} (P({\sU_m/\sV_m}), \alb(\sX)) \oplus \alb(\sX)[n](k)\\
		& \cong  \Hom_{k} (\alb(\sX)^m , \alb(\sX)) \oplus \alb(\sX)[n](k)\\
		& \cong (\en_{k}(\alb(\sX)))^m    \oplus \alb(\sX)[n](k).
	\end{align*}
	We denote by $\tilde{Q}_i$  the image of $\tilde{P}_i$'s by  the Albanese map
	$\tilde{\alpha}: \tilde{\sX} \rightarrow \alb(\tilde{\sX})$  for  $i=1,\ldots, m$. Then,
	by tracing back the above isomorphisms, one can  see that the points
	$\tilde{Q}_i$'s
	form a  subset of independent generators for the Mordell-Weil group $\alb(\tilde{\sX})(\sK)$. Hence, as $\Z$-modules,  we have
	$$\rk (\widetilde{\alb(\sX)} (\sK) )=\rk(\alb(\tilde{\sX}))(\sK)=  m \cdot \rk (\en_{k}(\alb(\sX))).$$
	Therefore, we have completed the proof of Theorem \ref{main1}.
	
\end{proof}

\

Next, we apply Theorem \ref{prym-abelian} to the case of dihedral covers of $\P^1$. The dihedral group of order $2n$ is the group with presentation
\[D_n=\langle \sigma, \tau\mid \sigma^n=\tau^2=1, \tau\sigma\tau=\sigma^{-1}\rangle.\] This group sits in the following exact sequence
\begin{equation}
	0\to\Z_n\to D_n\to \Z_2\to 0.
\end{equation}
A $D_n$-Galois cover $f:X\to Y$ corresponds to the function field extension $k(X)=k(Y)(u,z)$, where $u^2=f\in k(Y)$ and $z^n=g\in k(Y)(u)$, see \cite{CP}, Section 5. The map $f:X\to Y$ factors as $X\xrightarrow{p} W\xrightarrow{q}Y$, where $W=X/\langle\sigma\rangle$, $p$ is a $\Z_n$-cover and $q$ is a $\Z_2$-cover.\par If $Y=\P^1$, we know that $k(\P^1)=k(x)$. Consider a $D_n$-cover $C\to\mathbb{P}^1$ and let $\mathcal{C}_m=\prod_{i=1}^m C$, i.e., the product of $m$ copies of the same dihedral cover as before. We denote the image of $D_n$ under the diagonal embedding $D_n\hookrightarrow\prod_{i=1}^m D_n$ again by $D_n$. Set $\mathcal{Y}_m=\mathcal{C}_m/D_n$. By Proposition \ref{prym product}, we have that
\begin{equation}
	P(\mathcal{C}_m/\mathcal{Y}_m)=\prod_i P(C/\mathbb{P}^1)
\end{equation}
By Remark \ref{prym for curves}, $P(C/\mathbb{P}^1)={(\ker\nm_f)}^0$. However as
$\jac(\mathbb{P}^1)=0$, it follows that $P(C/\mathbb{P}^1)=\jac(C)$. Now Proposition \ref{prym product} gives that
\begin{equation}\label{product for P1}
	P(\mathcal{C}_m/\mathcal{Y}_m)=\prod_i
	P(C/\mathbb{P}^1)=(\jac(C))^m.
\end{equation}

Let $L$ be the function field of $\sC_m$. Then $L=k(x_1,\dots, x_m, u_1,\dots,u_m, z_1,\dots,z_m)$. Let $L^{\prime}=k(x_1,\dots, x_m, u_1,\dots,u_m)$. Set $L^{\prime}_1={L^{\prime}}^{\Z_2}=k(x_1,\dots, x_m, u_1u_{2},\dots,u_1u_m)$.  Setting $U_i=u_1u_{i+1}$, it holds that $U_i^2=f(x_1)f(x_{i+1})$ for $1,\dots, m-1$. So $L^{\prime}=L^{\prime}_1(u_1)$. It holds that $L=L^{\prime}(z_1,\dots,z_m)$. In this case, $L=L^{\prime}(s_1,\dots,s_m, z_1,\dots,z_m)$, where $s_i$ are the coordinates of $W_i$. $L^{\prime}_2={L}^{\Z_n}=k(t_1,\dots, t_m, z_1^{n-1}z_2,\dots,z_1^{n-1}z_m)$ and so $L=L^{\prime}_2(z_1)$. Setting $Z_i=z_1^{n-1}z_{i+1}$, it holds that $Z_i^n=g(s_1)^{n-1}g(s_{i+1})$ for $1,\dots, m-1$. Suppose $C_a$ is the twist of the dihedral cover $C$ by the extension $L\mid K$. Then $C_a$ can be defined by the equation
\begin{equation}
	\left\{
	\begin{array}{ll}
		f(x_1)U^2=f(x) \\
		g(s_1)Z^n=g(s),
	\end{array}
	\right.
\end{equation}
A computation shows that $C_a$ has the following rational points.

\[P_1=(x_1,s_1,1,1), P_i=(x_{i+1},s_{i+1}, U_i/f(x_1), Z_i/g(s_1)), i=1,\dots, m-1.\]

We define the 1-cocycle $Z^1(G,\aut(C))$ by $a_g=g$. Let $\jac(C)_a$ be the twist corresponding to this 1-cocycle. By applying \ref{rational isom} and \ref{product for P1}, it follows that
\begin{align*}
	\jac(C)_a(K)\cong\Hom_k(P(\mathcal{C}_m/\mathcal{Y}_m),\jac(C))\oplus \jac(C)[2n](k)\\
	\cong \Hom_k((\jac(C))^m,\jac(C))\oplus \jac(C)[2n](k)\\
	\cong \en_k(\jac(C))^m\oplus \jac(C)[2n](k).
\end{align*}
So that $\rk(\jac(C)_a(K))= m\cdotp \rk(\en_k(\jac(C)))$. Let $Q_i$ be the image of $P_i$ under the embedding $C_a\hookrightarrow \jac(C_a)$. Then, the $Q_i$ are among the generators of $\jac(C)_a(K)$.

\begin{remark}
	By similar arguments as in the proof of Theorem \ref{main1} and the above example, one can easily show that the same results hold for dihedral covers of the projective space $\P^{\ell}$ for for all $l\geq 1$. In fact the above results can be generalized to metacyclic Galois covers of $\P^{\ell}$.
\end{remark}


\begin{thebibliography}{00}
	\bibitem{B}    
	A. Beauville,\emph{Vari\'eti\'es de Prym et Jacobiennes intermediares.} Ann.scient. \'Ec. Norm. Sup.,10, (1977), 309-391.
	\bibitem{BS}
	A. Borel., J.P. Serre, \emph{Th\'eor\`emes de finitude en cohomolo-gie galoisienne.} Comment. Math. Helv.,39, (1964), 111-164.
	\bibitem{CP}
	F. Catenese, F. Perroni, \emph{Dihedral Galois covers of algebraic varieties and the simple cases.} Journal of Geometry and Physics 118 (2017), 67–93.
	\bibitem{Haz1}
	F. Hazama, \emph{On the Mordell-Weil group of certain abelian varieties defined over function fields.} J. Number Theory,37, (1991), 168-172.
	\bibitem{Haz}
	F. Hazama, \emph{Rational points on certain abelian varieties over function fields.} J. Number Theory,50, (1995), 278-285.
	\bibitem{LO}
	H. Lange, A. Ortega, \emph{Prym varieties of cyclic coverings.} Geom. Dedicata.,150, (2011), 391-403.
	\bibitem{M1}
	A. Mohajer, \emph{On the Prym map of Galois coverings.}  arXiv:2004.09678 (2020). To appear in Rocky Mountain Journal of Mathematics.
	\bibitem{M2}
	A. Mohajer, \emph{Rational points on abelian varieties over function fields and Prym varieties.} Arch. Math. 116, pp. 293-300 (2021). https://doi.org/10.1007/s00013-020-01550-4.
	\bibitem{RR}
	S. Recillas, R. Rodr\'iguez, \emph{Prym varieties and fourfold covers.} arXiv:math/0303155, 2003.
	\bibitem{Sal1}
	S. Salami, \emph{The rational points on certain abelian varieties over function fields.} J. of Number Theory,195, February (2019), 330-337.
	\bibitem{Sal2}
	S. Salami, \emph{Twists of the Albanese varieties of cyclic multipleplanes with large ranks over higher dimensionfunction fields.} Journal de Th\'eorie des Nombres de Bordeaux 32(2020), 861–876.
	%
	%
	\bibitem {Wang}   
	W. B. Wang:  \emph{On the twist of abelian varieties defined by the Galois extension of prime degree}, Journal of Algebra, 163  (3),  (1994), 813 - 818.
	
	\bibitem {Yun}
	G.  Yun: \emph{A note on finite abelian covers}, SCIENCE CHINA Mathematics,  Vol. 54 No. 7,  (2011): 1333–1342
\end{thebibliography}
\end{document}